\documentclass[12pt]{article}

\usepackage{amsmath}
\usepackage{amsfonts}
\usepackage{amsthm}
\usepackage{bigints}

\begin{document}

\newcommand{\A}{\mbox{${{{\cal A}}}$}}
\newcommand{\Mm}{{\cal{M}}^{\mu}}
\newcommand{\Mn}{{\cal{M}}^{\nu}}


\author{Attila Losonczi}
\title{Measures by means, means by measures}

\date{\today}

\newtheorem{thm}{\qquad Theorem}[section]
\newtheorem{prp}[thm]{\qquad Proposition}
\newtheorem{lem}[thm]{\qquad Lemma}
\newtheorem{cor}[thm]{\qquad Corollary}
\newtheorem{rem}[thm]{\qquad Remark}
\newtheorem{ex}[thm]{\qquad Example}
\newtheorem{df}[thm]{\qquad Definition}
\newtheorem{prb}{\qquad Problem}

\maketitle

\begin{abstract}

\noindent

We construct measure which determines a two variable mean in a very natural way. Using that measure we can extend the mean to infinite sets as well. E.g. we can calculate the geometric mean of any set with positive Lebesgue measure. We also study the properties and behavior of such generalized means that are obtained by a measure, and we provide some applications as well.

\noindent
\footnotetext{\noindent
AMS (2010) Subject Classifications: 28A10,28A20,26E60 \\

}

\end{abstract}

\section{Introduction}
This paper can be considered as a natural continuation of the investigations started in \cite{lamis} and \cite{lamisii} where we started to build the theory of means on infinite sets. An ordinary mean is for calculating the mean of two (or finitely many) numbers. This can be extended in many ways in order to get a more general concept where we have a mean on some infinite subsets of $\mathbb{R}$. The various general properties of such means, the relations among those means were studied thoroughly in \cite{lamis} and \cite{lamisii}. 

In this paper we look for the answer the following question. How can one generalize two variable means (e.g. the geometric mean) to Lebesgue measurable subsets of $\mathbb{R}$. I.e. can we calculate the geometric mean of a set with positive Lebesgue measure? We are going to answer this question on measure theoretic ground.

In the first part of the paper we investigate means that are created by measures on $\mathbb{R}$. We enumerate many properties of such means and we also study uniqueness.

In the second part of the paper we fulfill our main aim that is to find a measure that generates a given two variable mean. We prove that under some basic smoothness conditions the generating measure always exists. 

Based on that result we provide an interesting application to two variable means. We show that such two variable means ${\cal{K}}(a,b)$ are determined by the values taken at points $(1,x)$ only,  i.e. ${\cal{K}}(a,b)$ can be calculated from $f(x)={\cal{K}}(1,x)$ in a generic way. 

We also investigate some alternative ways how one can generate the given two variable mean. 

One may ask whether some properties of a two variable mean are inherited to the associated generalized mean. In this respect we show that the inequality between the arithmetic and geometric means remains valid for the associated generalized means too.

In the last section we analyze the behavior of such means in infinity and show a sufficient condition for a mean approaching the arithmetic mean in infinity. 

\subsection{Basic notions and notations}
For $K\subset\mathbb{R},\ y\in\mathbb{R}$ we use the notation $K^{y-}=K\cap(-\infty,y],\ K^{y+}=K\cap[y,+\infty).$

If $H\subset\mathbb{R},x\in\mathbb{R}$, then set $H+x=\{h+x:h\in H\}$.
We use the convention that this operation $+$ has to be applied prior to the set theoretical operations, e.g. $H\cup K\cup L+x=H\cup K\cup (L+x)$.

\medskip

Let ${\cal{K}}$ be a \textbf{two variable mean}.
${\cal{K}}$ is called symmetric if ${\cal{K}}(a,b)={\cal{K}}(b,a)$. It is strictly internal if $a<{\cal{K}}(a,b)<b$ whenever $a<b$. It is called continuous if it is a continuous 2-variable function of $a$ and $b$.

\medskip

Now we recall some basic notions and notations for generalized means from \cite{lamis} and \cite{lamisii}. Please note that here we are dealing with bounded sets only.

A \textbf{generalized mean} is a function ${\cal{K}}:C\to \mathbb{R}$ where $C\subset \mathrm{P}(\mathbb{R})$ consists of some (finite or infinite) subsets of $\mathbb{R}$ and $\inf H\leq {\cal{K}}(H)\leq\sup H$ holds for all $H\in C$.

In the sequel we will mostly omit the adjective ''generalized'' as from the context it will be clear.

\medskip

A mean ${\cal{K}}$ is called \textbf{internal} if $\inf H\leq {\cal{K}}(H)\leq\sup H.$ It is \textbf{strongly internal} if $\varliminf H\leq {\cal{K}}(H)\leq\varlimsup H$ where $\varliminf H=\min H',\varlimsup H=\max H'$ (where $H'$ is the accumulation points of $H$).

$\cal{K}$ has property \textbf{strict strong internality} if it is strongly internal and $\varliminf H<{\cal{K}}(H)<\varlimsup H$ whenever $H$ has at least 2 accumulation points.

$\cal{K}$ is \textbf{disjoint-monotone} if $H_1\cap H_2=\emptyset,{\cal{K}}(H_1)\leq{\cal{K}}(H_2)$, then ${\cal{K}}(H_1)\leq{\cal{K}}(H_1\cup H_2)\leq {\cal{K}}(H_2)$.

${\cal{K}}$ is \textbf{union-monotone} if $B\cap C=\emptyset$,
${\cal{K}}(A)\leq{\cal{K}}(A\cup B),{\cal{K}}(A)\leq{\cal{K}}(A\cup C)$ implies ${\cal{K}}(A)\leq{\cal{K}}(A\cup B\cup C)$ 
and 
${\cal{K}}(A\cup B)\leq{\cal{K}}(A),{\cal{K}}(A\cup C)\leq {\cal{K}}(A)$ implies ${\cal{K}}(A\cup B\cup C)\leq{\cal{K}}(A)$.
Moreover if any of the inequalities on the left hand side is strict, then so is the inequality on the right hand side.

${\cal{K}}$ is \textbf{bi-slice-continuous} if $H\in \mathrm{Dom}({\cal{K}})$, then $H^{\varliminf H+},H^{\varlimsup H-}\in \mathrm{Dom}({\cal{K}})$ and 
$f(x,y)={\cal{K}}(H^{x-}\cup H^{y+})$ is continuous where $\mathrm{Dom}(f)=\{(x,y):H^{x-}\cup H^{y+}\in \mathrm{Dom}({\cal{K}})\}$.

${\cal{K}}$ is \textbf{Cantor-continuous} if $H_i\in \mathrm{Dom}({\cal{K}}), H_{i+1}\subset H_i,\ \cap_{n=1}^{\infty}H_i\in \mathrm{Dom}({\cal{K}})$ implies that ${\cal{K}}(H_i)\to {\cal{K}}(\cap_{n=1}^{\infty}H_i)$. 

\medskip

Throughout this paper $\lambda$ will denote the Lebesgue measure.

If $H$ is bounded, Lebesgue measurable, $\lambda(H)>0$, then \[\mathrm{Avg}(H)=\frac{\int\limits_H x\ d\lambda}{\lambda(H)}.\]

If $f:\mathbb{R}\to\mathbb{R}$ is an increasing continuous function, then let $\mu_f$  be the Lebesgue-Stieltjes measure belonging to $f$, i.e. the Carathéodory extension of $\mu([a,b))=f(b)-f(a)\ (a<b)$.

\medskip

Each measure considered in this paper will be from a special class of Borel measures that we define here.

\begin{df}A measure $\mu$ is called \textbf{p-Borel measure} if it is a Borel measure on some interval (finite or infinite) of the real line $\mathbb{R}$ that satisfies two conditions: 
\par (1) if $H\subset\mathbb{R}$ is bounded and measurable, then $\mu(H)<+\infty$,
\par (2) if $I$ is a non degenerative interval, then $0<\mu(I)$.\qed
\end{df}

Let us remark that if $H\subset\mathbb{R}$ is bounded, then $\mu|_H$ is absolutely continuous with respect to $\lambda$ iff it is $\varepsilon-\delta$ absolutely continuous ($\forall\varepsilon>0\ \exists\delta>0$ such that $\lambda(K)<\delta$ implies that $\mu(K)<\varepsilon$).

\section{Means by measures}

We now define the main notion of the paper, that is how a measure determines naturally a generalized mean.

\begin{df}Let $I\subset\mathbb{R}$ be an interval (finite or infinite). Let $\mu$ be a p-Borel measure on $I$. Let $H$ be a bounded $\mu$-measurable set such that $0<\mu(H)<+\infty$. Set 
\[{\cal{M}}^{\mu}(H)=\frac{\int\limits_Hx d\mu}{\mu(H)}.\tag*{\qed}\]
\end{df}

Now our aim will be to enumerate many properties of this generalized mean.

\begin{prp}${\cal{M}}^{\mu}$ is internal. \qed
\end{prp}

\begin{prp}If $\mu(H)=0$ whenever $H$ is finite, then $\Mm$ is strongly-internal.
\end{prp}
\begin{proof} By \cite{lamisii} Proposition 2 it is enough to prove finite independence and internality. The condition is equivalent to finite independence.
\end{proof}

\begin{lem}\label{ldist0}Let $H_1,H_2\in \mathrm{Dom}(\Mm)$, $\mu((H_1-H_2)\cup (H_2-H_1))=0$. Then $\Mm(H_1)=\Mm(H_2)$.
\end{lem}
\begin{proof}Clearly \[|\Mm(H_1)-\Mm(H_2)|=\left\lvert\frac{\int\limits_{H_1} x d\mu}{\mu(H_1)}-\frac{\int\limits_{H_2} x d\mu}{\mu(H_2)}\right\lvert=
\left\lvert\frac{\int\limits_{H_1} x d\mu}{\mu(H_1)}-\frac{\int\limits_{H_1} x d\mu+\int\limits_{H_2-H_1} x d\mu}{\mu(H_2)}\right\lvert\leq\]
\[\left\lvert\int\limits_{H_1} x d\mu\right\lvert\left\lvert\frac{1}{\mu(H_1)}-\frac{1}{\mu(H_2)}\right\lvert+\left\lvert\frac{\int\limits_{H_2-H_1} x d\mu}{\mu(H_2)}\right\lvert=
\left\lvert\int\limits_{H_1} x d\mu\right\lvert\left\lvert\frac{\mu(H_2)-\mu(H_1)}{\mu(H_1)\mu(H_2)}\right\lvert+\left\lvert\frac{\int\limits_{H_2-H_1} x d\mu}{\mu(H_2)}\right\lvert=0.\] 
\end{proof}

\begin{prp}If $\mu(H)=0$ whenever $H$ is countable, then $\Mm$ is strict-strong-internal.
\end{prp}
\begin{proof} Let $H\in \mathrm{Dom}(\Mm), a=\varliminf H,b=\varlimsup H$. By $\mu(H^{a-})=\mu(H^{b+})=0$ we get that $\Mm([a,b])=\Mm(H)$. Obviously there is $c\in (a,b)$ such that $\mu(H\cap[c,b])>0$. Let $H_1=H\cap[a,c),H_2=H\cap[c,b]$. Then $\Mm(H)\geq \frac{\mu(H_1)a+\mu(H_2)c}{\mu(H)}=\frac{\mu(H_1)}{\mu(H)}a+\frac{\mu(H_2)}{\mu(H)}c>a$ because it is a weighted average and $\frac{\mu(H_2)}{\mu(H)}>0$.

The other inequality can be shown similarly.
\end{proof}

Exactly the same way one can show:

\begin{prp}Let $\mu$ is absolutely continuous with respect to $\lambda$, $H\in \mathrm{Dom}(\Mm)$. Let $a=\sup\{x\in\mathbb{R}:\lambda(H^{x-})=0\}, b=\inf\{x\in\mathbb{R}:\lambda(H^{x+})=0\}$ (cf. \cite{lamis} Definition 4). Then $a<\Mm(H)<b$. \qed
\end{prp}

From the countable additive property of the integral we immediately get that the mean is a natural weighted average of values on disjoint sets.

\begin{prp}\label{pmmwa}Let $H,H_i\in \mathrm{Dom}(\Mm)\ (i\in\mathbb{N}), H_i\cap H_j=\emptyset\ (i\ne j), H=\bigcup\limits_{i=1}^{\infty}H_i$. Then
\[\Mm(H)=\frac{\sum\limits_{i=1}^{\infty}\mu(H_i)\Mm(H_i)}{\sum\limits_{i=1}^{\infty}\mu(H_i)}.\tag*{\qed}\]
\end{prp}

\begin{prp}\label{pmmcc}Let $H_n,H\in \mathrm{Dom}(\Mm)$, $H=\bigcap\limits_{n=1}^{\infty}H_n$. Then $\Mm(H_n)\to \Mm(H)$, i.e. $\Mm$ is Cantor-continuous. 
\end{prp}
\begin{proof}Clearly $\displaystyle|\Mm(H_n)-\Mm(H)|=\left\lvert\frac{\int\limits_{H_n}x d\mu}{\mu(H_n)}-\frac{\int\limits_{H}x d\mu}{\mu(H)}\right\lvert=$
$$\left\lvert\frac{\mu(H)\left(\int\limits_{H_n}x d\mu-\int\limits_{H}x d\mu\right)+(\mu(H)-\mu(H_n))\int\limits_{H}x d\mu}{\mu(H_n)\mu(H)}\right\lvert=$$
$$\left\lvert\frac{\mu(H)\int\limits_{H_n-H}x d\mu+(\mu(H)-\mu(H_n))\int\limits_{H}x d\mu}{\mu(H_n)\mu(H)}\right\lvert\leq$$
$$\left\lvert\frac{\mu(H)\mu(H_n-H)\sup H_1}{\mu(H)^2}\right\lvert + 
\left\lvert\frac{(\mu(H)-\mu(H_n))\mu(H)\sup H}{\mu(H)^2}\right\lvert
\to 0$$
using that $\mu(H_n)\to\mu(H)$.
\end{proof}

\begin{df}Let ${\cal M}:{\cal H}\to\mathbb{R}$ be a set-function where ${\cal H}\subset \mathrm{P}(\mathbb{R})$. Let ${\cal H}'\subset {\cal H}$. We say that ${\cal M}|_{{\cal H}'}$
\textbf{determines} ${\cal M}$ if ${\cal N}:{\cal H}\to\mathbb{R}$ is given and ${\cal N}|_{{\cal H}'}={\cal M}|_{{\cal H}'}$ implies that ${\cal N}={\cal M}$.\qed
\end{df}

\begin{ex}It is well known that if $\mu$ is a Borel measure, then $\mu|_{{\cal H}'}$ determines $\mu$ where ${\cal H}'=\{(a,b):a,b\in\mathbb{R},\ a<b\}$.\qed
\end{ex}

Now we can show that the same is true for $\Mm$.

\begin{thm}\label{pMmdetbyvoi}If $\mu$ is a p-Borel measure, then $\Mm|_{{\cal H}'}$ determines $\Mm$ where ${\cal H}'=\{(a,b):a,b\in\mathbb{R},\ a<b\}$. 
\end{thm}
\begin{proof}For a given Borel set $H\subset\mathbb{R}$, by \cite{billings} 12.7 there is a sequence of sets $(H_n)$ such that $H\subset\bigcap\limits_{n=1}^{\infty}H_n,$ $\mu(H)=\mu\left(\bigcap\limits_{n=1}^{\infty}H_n\right)$ and $H_n$ is a countable union of disjoint open intervals. By Lemma \ref{ldist0}, Proposition \ref{pmmwa} and \ref{pmmcc} we get that $\Mm(H)$ is uniquely determined by $\langle\Mm((a,b)):a<b\rangle$. 
\end{proof}

We can naturally derive a two variable mean from $\Mm$.

\begin{df}Let $\Mm$ be given such that $[a,b]\in \mathrm{Dom}(\Mm),\ (a,b\in\mathbb{R}, a<b)$. Then set 
\[{\cal M}^{(\mu)}(a,b)={\cal M}^{(\mu)}(b,a)=\Mm([a,b])\text{\ \ and\ \ }{\cal M}^{(\mu)}(a,a)=a.\tag*{\qed}\]
\end{df}

\begin{ex}Let $f(x)=x^2,\mu=\mu_f$ defined on subsets of $\mathbb{R}^+\cup\{0\}$. Then ${\cal M}^{(\mu)}(a,b)=\frac{2}{3}\frac{a^2+ab+b^2}{a+b}$.
\end{ex}

\begin{ex}Let $f(x)=e^x,\mu=\mu_f$. Then ${\cal M}^{(\mu)}(a,b)=\frac{be^b-ae^a}{e^b-e^a}-1$.
\end{ex}

\begin{thm}If $H\in \mathrm{Dom}(\Mm)$, then $\Mm(H)$ is determined by all ${\cal M}^{(\mu)}(a,b)\ (a<b)$. 
\end{thm}
\begin{proof}By Theorem \ref{pMmdetbyvoi}, we only need that $\Mm\big((a,b)\big)$ is determined by all $\Mm([a',b'])\ (a'<b')$, which is well known.
\end{proof}

\begin{prp}Let $\Mm$ be given such that $[a,b]\in \mathrm{Dom}(\Mm),\ (a,b\in\mathbb{R}, a<b)$. Then ${\cal M}^{(\mu)}(a,b)$ is strictly-internal (i.e. it is a mean) and symmetric.
\end{prp}
\begin{proof}Property symmetric is trivial as we defined it that way.
\par It is strictly-internal because suppose otherwise
\[\frac{\int\limits_a^bx\ d\mu}{\mu(H)}=\frac{\int\limits_a^bx\ d\mu}{\int\limits_a^b1\ d\mu}=a.\]
This would imply that
$\int\limits_a^bx\ d\mu=\int\limits_a^ba\ d\mu$
hence
$\int\limits_a^bx-a\ d\mu=0$
which gives that $\mu([\frac{a+b}{2},b])=0$ which is a contradiction because all non-degenerative interval is in the domain of $\Mm$ i.e. $\mu$ does not vanish on them. 
\par If $\frac{\int\limits_a^bx\ d\mu}{\mu(H)}=b$, then a similar argument works.
\end{proof}

\begin{ex}${\cal M}^{(\mu)}$ is not necessarily continuous.
\par If $H\subset\mathbb{R}$ is a bounded Borel set, then let
\[\mu(H)=\begin{cases}
\lambda(H)&\text{if }2\notin H\\
\lambda(H)+1&\text{if }2\in H.
\end{cases}\]
Obviously $\mu$ is a measure on the Borel sets, and $f(y)={\cal M}^{(\mu)}(1,y)\ (y>1)$ is not continuous at $2$, because
$\lim\limits_{y\to 2-0}{\cal M}^{(\mu)}(1,y)=1.5$
while ${\cal M}^{(\mu)}(1,2)=\frac{1.5+1}{1+1}=1.25$.\qed
\end{ex}

\begin{prp}Let $\Mm$ be given such that $[a,b]\in \mathrm{Dom}(\Mm),\ (a,b\in\mathbb{R}, a<b)$ and let $f(y)={\cal M}^{(\mu)}(1,y)\ (y>1)$. Then $f$ is a.e. continuous.
\end{prp}
\begin{proof}Let $\nu(H)=\int\limits_Hx\ d\mu$. Evidently $\nu$ is absolute continuous w.r.t. $\mu$ and then it has a Radon-Nikodym derivative $g$ and $g(x)=x$ holds almost everywhere. It can be readily seen that where $g(x)=x$ holds, $f$ is continuous.
\end{proof}

Now we continue the investigation of properties of $\Mm$.

\begin{prp}$\Mm$ is disjoint-monotone. 
\end{prp}
\begin{proof} Let $H_1\cap H_2=\emptyset, \Mm(H_1)\leq \Mm(H_2)$. 
Then $\displaystyle\Mm(H_1\cup H_2)=\frac{\mu(H_1)\Mm(H_1)+\mu(H_2)\Mm(H_2)}{\mu(H_1\cup H_2)}\leq\frac{\mu(H_1)\Mm(H_2)+\mu(H_2)\Mm(H_2)}{\mu(H_1)+\mu(H_2)}=\Mm(H_2)$. The other inequality is similar.
\end{proof}

\begin{prp}$\Mm$ is union-monotone.
\end{prp}
\begin{proof} Let $B\cap C=\emptyset,\Mm(A)\leq \Mm(A\cup B),\Mm(A)\leq \Mm(A\cup C)$. Obviously we can assume that $A\cap B=A\cap C=\emptyset$. We know that $\displaystyle\Mm(A)\leq\frac{\mu(A)\Mm(A)+\mu(B)\Mm(B)}{\mu(A)+\mu(B)}$ and $\displaystyle\Mm(A)\leq\frac{\mu(A)\Mm(A)+\mu(C)\Mm(C)}{\mu(A)+\mu(C)}$. 

Then 
\[\Mm(A\cup B\cup C)=\frac{\mu(A)\Mm(A)+\mu(B)\Mm(B)+\mu(C)\Mm(C)}{\mu(A)+\mu(B)+\mu(C)}
\geq\]
\[\frac{(\mu(A)+\mu(B))\Mm(A)+\mu(C)\Mm(C)}{\mu(A)+\mu(B)+\mu(C)}
=\frac{\mu(A)\Mm(A)+\mu(C)\Mm(C)+\mu(B)\Mm(A)}{\mu(A)+\mu(B)+\mu(C)}\geq\]
\[\frac{(\mu(A)+\mu(C))\Mm(A)+\mu(B)\Mm(A)}{\mu(A)+\mu(B)+\mu(C)}=\Mm(A).\]

Clearly if $\Mm(A)<\Mm(A\cup B)$ also holds, then we get that $\Mm(A)<\Mm(A\cup B\cup C)$.

The opposite inequalities can be handled similarly.
\end{proof}

\begin{lem}\label{lmmc}Let $I$ be a bounded interval, $\mu$ be a p-Borel measure on $I$. Then $\forall H_1\in \mathrm{Dom}(\Mm)\ \forall\varepsilon>0\ \exists\delta>0$ such that ${\mu\big((H_1-H_2)\cup (H_2-H_1)\big)<\delta}$ and $H_2\in \mathrm{Dom}(\Mm)$ implies that $|\Mm(H_1)-\Mm(H_2)|<\varepsilon$.
\end{lem}
\begin{proof}Clearly
 $|\Mm(H_1)-\Mm(H_2)|=\left\lvert\frac{\int\limits_{H_1} x d\mu}{\mu(H_1)}-\frac{\int\limits_{H_2} x d\mu}{\mu(H_2)}\right\lvert=$

$\left\lvert\frac{\int\limits_{H_1-H_2} x d\mu+\int\limits_{H_1\cap H_2} x d\mu}{\mu(H_1)} - \frac{\int\limits_{H_1\cap H_2} x d\mu+\int\limits_{H_2-H_1} x d\mu}{\mu(H_2)}\right\lvert\leq$

$\left\lvert\frac{\int\limits_{H_1-H_2} x d\mu}{\mu(H_1)}\right\lvert + \left\lvert\int\limits_{H_1\cap H_2} x d\mu\right\lvert\left\lvert\frac{1}{\mu(H_1)}-\frac{1}{\mu(H_2)}\right\lvert + \left\lvert\frac{\int\limits_{H_2-H_1} x d\mu}{\mu(H_2)}\right\lvert=$

$\left\lvert\frac{\int\limits_{H_1-H_2} x d\mu}{\mu(H_1)}\right\lvert + \left\lvert\int\limits_{H_1\cap H_2} x d\mu\right\lvert\left\lvert\frac{\mu(H_2)-\mu(H_1)}{\mu(H_1)\mu(H_2)}\right\lvert+\left\lvert\frac{\int\limits_{H_2-H_1} x d\mu}{\mu(H_2)}\right\lvert$. 

Clearly $\int\limits_{H_1- H_2} x d\mu<\delta\sup H_1$, $\int\limits_{H_1\cap H_2} x d\mu<\mu(H_1)\sup H_1=K_1, \int\limits_{H_2-H_1} x d\mu<\delta\sup H_2\leq\delta\sup I$.

If $\delta<\frac{\mu(H_1)}{2}$, then
\[|\Mm(H_1)-\Mm(H_2)|<\frac{\delta\sup H_1}{\mu(H_1)} + K_1\frac{2\delta}{\mu(H_1)^2} + \frac{2\delta\sup I}{\mu(H_1)}<\varepsilon\]
showing that $\delta$ can be chosen.
\end{proof}

\begin{cor}Let $I$ be a bounded interval, $\mu$ be a p-Borel measure on $I$. If $\mathrm{Dom}(\Mm)$ is equipped with the pseudo-metric ${d_{\mu}(H_1,H_2)=\mu\big((H_1-H_2)\cup (H_2-H_1)\big)}$, then $\Mm$ is continuous according to $d_{\mu}$. \qed
\end{cor}

\begin{ex}This is obviously not true if $I$ is not bounded. See e.g. $\mu=\lambda, H_1=[0,1], \varepsilon=0.1, H_2=[0,1]\cup[\frac{1}{\delta},\frac{1}{\delta}+\delta]$. Then $\mathrm{Avg}(H_1)=0.5,\mathrm{Avg}(H_2)=\frac{0.5\cdot 1+(\frac{1}{\delta}+\frac{\delta}{2})\delta}{1+\delta}>0.75$.
\end{ex}

\begin{prp}If $\mu$ is absolutely continuous with respect to $\lambda$, then $\Mm$ is bi-slice-continuous.
\end{prp}
\begin{proof} We know that $\forall\varepsilon>0\ \exists\delta>0$ such that $\lambda(H)<\delta$ implies that $\mu(H)<\varepsilon$. Then apply Lemma \ref{lmmc}.
\end{proof}

Our next aim is to investigate inequalities between means.

\begin{lem}\label{lhut}Let $H,H_i\in \mathrm{Dom}(\Mm)\ (i\in\mathbb{N}), H_i\cap H_j=\emptyset\ {(i\ne j)}, {H=\cup_{i=1}^{\infty}H_i}$. Then
$\lim\limits_{n\to\infty}\Mm(\cup_{i=1}^{n}H_i)=\Mm(H)$.
\end{lem}
\begin{proof} It is enough to refer to Lemma \ref{lmmc}.
\end{proof}

\begin{prp}\label{pineq}Let $\mu,\nu$ be p-Borel measures on an interval $I$. Assume that  if $I_i\subset I\ (1\leq i\leq n, n\in\mathbb{N})$ are disjoint bounded open intervals, then $\Mm(\cup_{i=1}^n I_i)\leq{\cal{M}}^{\nu}(\cup_{i=1}^n I_i)$. 
Then $\Mm(H)\leq{\cal{M}}^{\nu}(H)\ \forall H\in \mathrm{Dom}(\Mm)\cap \mathrm{Dom}({\cal{M}}^{\nu})$.
\end{prp}
\begin{proof} By Lemma \ref{lhut} it is true for countably many intervals too i.e. it is valid for any bounded open set. Finally if $H\in \mathrm{Dom}(\Mm)\cap \mathrm{Dom}({\cal{M}}^{\nu})$, then Cantor-continuity and Lemma \ref{ldist0} yield that $\Mm(H)\leq{\cal{M}}^{\nu}(H)$.
\end{proof}

Now we present a sufficient condition which together with \ref{suffcondforcond2}, we will apply later in section \ref{s3}.

\begin{prp}\label{suffcondforineq}Let $\mu,\nu$ be p-Borel measures on an interval $I$ and let us assume that the following two conditions hold.
\vspace{-0.3pc}\begin{enumerate}\setlength\itemsep{-0.3em}
\item If $J\subset I$ is a bounded open interval, then $\Mm(J)\leq{\cal{M}}^{\nu}(J)$.
\item If $J,K$ are bounded open intervals such that $\sup J\leq\inf K$, then ${\frac{\nu(J)}{\mu(J)}\leq\frac{\nu(K)}{\mu(K)}}$.
\end{enumerate}
Then $\Mm(H)\leq\Mn(H)\ \forall H\in \mathrm{Dom}(\Mm)\cap \mathrm{Dom}({\cal{M}}^{\nu})$.
\end{prp}
\begin{proof} First let us observe that condition 2 simply implies that if $I_i\subset I\ (1\leq i\leq n, n\in\mathbb{N})$ are disjoint bounded open intervals and $\sup I_i\leq\inf K\ (\forall i)$, then $$\frac{\nu(\cup_{i=1}^n I_i)}{\mu(\cup_{i=1}^n I_i)}\leq\frac{\nu(K)}{\mu(K)}.$$

On the same assumptions by \ref{pineq} we have to show $\Mm(\cup_{i=1}^n I_i)\leq{\cal{M}}^{\nu}(\cup_{i=1}^n I_i)$. It holds for $n=1$ by condition 1. We go on by induction. Suppose it is true for $n-1$. Let $\sup I_i\leq\inf I_n\ (1\leq i\leq n-1)$.

Using Proposition \ref{pmmwa} we have
\[\Mm(\cup_{i=1}^n I_i)=\frac{\mu(\cup_{i=1}^{n-1} I_i)\Mm(\cup_{i=1}^{n-1} I_i) + \mu(I_n)\Mm(I_n)}{\mu(\cup_{i=1}^{n-1} I_i)+\mu(I_n)}\leq\]
\[\frac{\mu(\cup_{i=1}^{n-1} I_i)\Mn(\cup_{i=1}^{n-1} I_i) + \mu(I_n)\Mn(I_n)}{\mu(\cup_{i=1}^{n-1} I_i)+\mu(I_n)}.\]
It is enough to prove that 
$$\frac{\mu(\cup_{i=1}^{n-1} I_i)\Mn(\cup_{i=1}^{n-1} I_i) + \mu(I_n)\Mn(I_n)}{\mu(\cup_{i=1}^{n-1} I_i)+\mu(I_n)}\leq
\frac{\nu(\cup_{i=1}^{n-1} I_i)\Mn(\cup_{i=1}^{n-1} I_i) + \nu(I_n)\Mn(I_n)}{\nu(\cup_{i=1}^{n-1} I_i)+\nu(I_n)}$$
and that is equivalent to
$$\mu(\cup_{i=1}^{n-1} I_i)\Mn(\cup_{i=1}^{n-1} I_i)\nu(I_n) +  \mu(I_n)\Mn(I_n)\nu(\cup_{i=1}^{n-1} I_i)\leq$$
$$\nu(\cup_{i=1}^{n-1} I_i)\Mn(\cup_{i=1}^{n-1} I_i)\mu(I_n) +  \nu(I_n)\Mn(I_n)\mu(\cup_{i=1}^{n-1} I_i)$$
and
$$0\leq \Big(\Mn(I_n) - \Mn(\cup_{i=1}^{n-1} I_i)\Big) \Big(\nu(I_n)\mu(\cup_{i=1}^{n-1} I_i) - \mu(I_n)\nu(\cup_{i=1}^{n-1} I_i)  \Big)$$

But the first term is obviously positive and by the consequence of condition 2 
$\nu(I_n)\mu(\cup_{i=1}^{n-1} I_i) \geq \mu(I_n)\nu(\cup_{i=1}^{n-1} I_i)$ is valid as well.
\end{proof}

\begin{prp}\label{suffcondforcond2}Let $f,g$ be increasing differentiable functions. If $\frac{g'(x)}{f'(x)}$ is increasing, then condition 2 (in \ref{suffcondforineq}) holds for $\mu_f,\mu_g$.
\end{prp}
\begin{proof} Let $J=(a,b), K=(c,d), b\leq c$. We have to show that
$$\frac{g(b)-g(a)}{f(b)-f(a)}\leq\frac{g(d)-g(c)}{f(d)-f(c)}.$$
By Cauchy's mean value theorem there are $\alpha\in(a,b)$ and $\beta\in(c,d)$ such that 
$$\frac{g(b)-g(a)}{f(b)-f(a)}=\frac{g'(\alpha)}{f'(\alpha)},\frac{g(d)-g(c)}{f(d)-f(c)}=\frac{g'(\beta)}{f'(\beta)}$$
and $\frac{g'(\alpha)}{f'(\alpha)}\leq\frac{g'(\beta)}{f'(\beta)}$ by assumption.
\end{proof}

Let us investigate uniqueness.

\begin{thm}Let $\Mm=\Mn$. Then there is $c\in\mathbb{R}, c>0$ such that $\nu=c\mu$.
\end{thm}
\begin{proof} Let $A,B\in \mathrm{Dom}(\Mm), A\cap B=\emptyset$. Then $\Mm(A\cup B)=\Mn(A\cup B)$. By \ref{pmmwa}
$$\frac{\mu(A)\Mm(A)+\mu(B)\Mm(B)}{\mu(A)+\mu(B)}=\frac{\nu(A)\Mn(A)+\nu(B)\Mn(B)}{\nu(A)+\nu(B)}.$$
By $\Mm(A)=\Mn(A),\Mm(B)=\Mn(B)$
$$\frac{\mu(A)\Mm(A)+\mu(B)\Mm(B)}{\mu(A)+\mu(B)}=\frac{\nu(A)\Mm(A)+\nu(B)\Mm(B)}{\nu(A)+\nu(B)}.$$
Then
$$\mu(A)\Mm(A)\nu(B)+\mu(B)\Mm(B)\nu(A)=\nu(A)\Mm(A)\mu(B)+\nu(B)\Mm(B)\mu(A)$$ and
$$\Big(\Mm(A) - \Mm(B)\Big)\Big(\mu(A)\nu(B) - \nu(A)\mu(B)\Big)=0.$$
If $\Mm(A) \ne \Mm(B), A\cap B=\emptyset$, then $$\frac{\nu(A)}{\mu(A)}=\frac{\nu(B)}{\mu(B)}$$ has to hold.

Let $I=(0,1), c=\frac{\nu(I)}{\mu(I)}$. If $H\in \mathrm{Dom}(\Mm)$, then let $J$ be an interval such that $I\cap J=\emptyset,\ \sup H<\inf J$. Then $H\cap J=\emptyset$ and $\Mm(I) \ne \Mm(J),\Mm(H) \ne \Mm(J)$ hold. We get that $\frac{\nu(I)}{\mu(I)}=\frac{\nu(J)}{\mu(J)}$ and $\frac{\nu(J)}{\mu(J)}=\frac{\nu(H)}{\mu(H)}$. I.e. $\nu(H)=c\mu(H)$.
\end{proof}

\section{Measures by means}\label{s3}

Let a two variable mean ${\cal{K}}$ be given that is just for calculating the mean of two numbers. Can we extend this mean somehow to some subsets of $\mathbb{R}$? We may have many options for doing so. But now we are going to approach this from measure theory.

We know that $$\frac{\int\limits_a^b x d\lambda}{\lambda([a,b])}=\frac{a+b}{2}$$ that is the arithmetic mean.
Therefore we can try to look for a measure $\mu$ such that $\frac{\int\limits_a^b x d\mu}{\mu([a,b])}={\cal{K}}(a,b)$ where $a,b\in\mathbb{R}$ and $[a,b]\in \mathrm{Dom}({\cal{K}})$.

\begin{thm}\label{tmbm}Let ${\cal{K}}$ be a two variable mean that is symmetric, strictly internal, continuous and $\frac{\partial{\cal{K}}(x,y)}{\partial y}$ exists for all $(x,y)\in \mathrm{Dom}({\cal{K}})$ and it is continuous. Then there exists a measure $\mu$ that is absolutely continuous with respect to $\lambda$ such that ${\cal{K}}(a,b)=\Mm([a,b])$.
\end{thm}
\begin{proof}
Let us look for $\mu$ in the form $\mu=\mu_f$ where $f$ is an increasing differentiable function. Then $\mu([a,b])=f(b)-f(a)$ and $f'=\frac{d\mu}{d\lambda}$.

Then $${\cal{K}}(a,b)=\frac{\int\limits_a^b x d\mu_f}{\mu_f([a,b])}=\frac{\int\limits_a^b xf' d\lambda}{f(b)-f(a)}=
\frac{[xf-F]_a^b}{f(b)-f(a)}=\frac{bf(b)-af(a)-(F(b)-F(a))}{f(b)-f(a)}$$ where $F$ is a primitive function of $f$. We can assume that there is a point $a$ such that $f(a)=F(a)=0$ because $f$ and $f+c$, $F$ and $F+d$ have the same effect ($c,d\in\mathbb{R}$). Let us suppose that $a=1$ i.e. $f(1)=F(1)=0$. Then we get ${\cal{K}}(1,x)=\frac{xf(x)-F(x)}{f(x)}=x-\frac{F(x)}{f(x)}\ (x\ne1)$.

Let us also note that $f$ and $cf$ (or $F$ and $cF$) give the same result ($c\in\mathbb{R}$) i.e. if $f$ solves our problem, then so does $cf$ because ${\cal{M}}^{\mu_f}={\cal{M}}^{\mu_{cf}}$.

We can write $\frac{F(x)}{f(x)}=x-{\cal{K}}(1,x)$.
Equivalently $\frac{f(x)}{F(x)}=\frac{1}{x-{\cal{K}}(1,x)}$ as ${\cal{K}}$ is strictly internal we do not divide here by 0.

Then $(\log F(x))'=\frac{F'(x)}{F(x)}=\frac{f(x)}{F(x)}=\frac{1}{x-{\cal{K}}(1,x)}$

Let $\varepsilon>0,b>1+\varepsilon$. (If $b<1$, then we can handle that similarly.)

$\int\limits_{1+\varepsilon}^b(\log F(x))'=\log F(b)-\log F(1+\varepsilon)=\int\limits_{1+\varepsilon}^b\frac{1}{x-{\cal{K}}(1,x)}dx$

The integral on the right hand side exists because $[1+\varepsilon,b]$ is compact, $x-{\cal{K}}(1,x)$ is continuous hence it takes its minimum but it is $>0$ since ${\cal{K}}$ is strictly internal. 
Set $c=F(1+\varepsilon)$. 
Then
\[F(b)=ce^{\int\limits_{1+\varepsilon}^b\frac{1}{x-{\cal{K}}(1,x)}dx}\]
but as we noted a constant factor can be abandoned hence let
\[F(b)=e^{\int\limits_{1+\varepsilon}^b\frac{1}{x-{\cal{K}}(1,x)}dx}.\]
Then we get
\[f(b)=\frac{1}{b-{\cal{K}}(1,b)}e^{\int\limits_{1+\varepsilon}^b\frac{1}{x-{\cal{K}}(1,x)}dx}=\frac{1}{b-{\cal{K}}(1,b)}F(b)\]
\[f'(b)=\frac{\frac{\partial{\cal{K}}(1,b)}{\partial b}}{(b-{\cal{K}}(1,b))^2}e^{\int\limits_{1+\varepsilon}^b\frac{1}{x-{\cal{K}}(1,x)}dx}=\frac{\frac{\partial{\cal{K}}(1,b)}{\partial b}}{(b-{\cal{K}}(1,b))^2}F(b)=\frac{\frac{\partial{\cal{K}}(1,b)}{\partial b}}{b-{\cal{K}}(1,b)}f(b)\]

We got $f$ and $F$  by using ${\cal{K}}(1,x)$ only.
Therefore we also have to check whether $f$ and $F$ fulfills our original request i.e. they work for ${\cal{K}}(a,b)$ as well.

$$\frac{\int\limits_a^b x d\mu_f}{\mu_f([a,b])}=\frac{\int\limits_1^b x d\mu_f-\int\limits_1^a x d\mu_f}{f(b)-f(a)}=
\frac{\mu_f([1,b]){\cal{K}}(1,b)-\mu_f([1,a]){\cal{K}}(1,a)}{f(b)-f(a)}=$$

$$\frac{bf(b)-1f(1)-(F(b)-F(1))-(af(a)-1f(1)-(F(a)-F(1)))}{f(b)-f(a)}=$$

\[\frac{bf(b)-af(a)-(F(b)-F(a))}{f(b)-f(a)}={\cal{K}}(a,b).\qedhere\]
\end{proof}

\begin{rem}We can simplify the conditions because if ${\cal{K}}(x,y)$ is continuously differentiable by its second variable, then it continuously differentiable by its first variable since ${\cal{K}}$ is symmetric, therefore ${\cal{K}}$ is totally differentiable hence continuous.
\end{rem}

As a consequence, we can prove that all values of those two variable means are determined by the values taken at points $(1,x)$, i.e. it is enough to know the values at points $(1,x)$, and from them, we can get any other value.

\begin{cor}Let ${\cal{K}}$ be a two variable mean that is symmetric, strictly internal, continuous and $\frac{\partial{\cal{K}}(x,y)}{\partial y}$ exists for all $(x,y)\in \mathrm{Dom}({\cal{K}})$ and it is continuous. Then ${\cal{K}}|_{{\cal H}'}$ determines ${\cal{K}}$ where ${\cal H}'=\{(1,x):x\in\mathbb{R}\}$.
\end{cor}
\begin{proof} Using the constructed $f$ in Theorem \ref{tmbm} we get 
$${\cal{K}}(a,b)=\frac{(f(b)-f(1)){\cal{K}}(1,b)-(f(a)-f(1)){\cal{K}}(1,a)}{f(b)-f(a)}$$
and $f$ is calculated by the function $x\mapsto{\cal{K}}(1,x)$.
\end{proof}

\begin{rem}If $f$ is an increasing differentiable function, $F$ is one of its primitive functions, then 
$${\cal{K}}(a,b)=\frac{bf(b)-af(a)-(F(b)-F(a))}{f(b)-f(a)}$$
define a symmetric, strictly internal, continuous two variable mean $(a<b)$.
\par For strictly internality we can give a direct proof when $f'\ne 0$.
By Cauchy's mean value theorem there is $\xi\in(a,b)$ such that ${\cal{K}}(a,b)=\frac{\xi f'(\xi)}{f'(\xi)}=\xi$.
\end{rem}

\begin{ex}Let ${\cal{K}}$ be the geometric mean: ${\cal{K}}(a,b)=\sqrt{ab}$. Then $F(x)=\frac{(\sqrt{x}-1)^2}{e^2},f(x)=\frac{1}{e^2}(1-\frac{1}{\sqrt{x}}),f'(x)=\frac{1}{2e^2}\frac{1}{x\sqrt{x}}$. Hence $\mu([a,b])=\mu_f([a,b])=\frac{1}{e^2}(\frac{1}{\sqrt{a}}-\frac{1}{\sqrt{b}})$.
\end{ex}
\begin{proof} We know that 
$$F(b)=e^{\int\limits_{1+\varepsilon}^b\frac{1}{x-{\cal{K}}(1,x)}dx}$$
hence we have to calculate 
\[\int\limits_{1+\varepsilon}^b\frac{1}{x-\sqrt{x}}dx=\int\limits_{1+\varepsilon}^b\frac{1}{\sqrt{x}(\sqrt{x}-1)}dx=
\int\limits_{1+\varepsilon}^b\frac{1}{\sqrt{x}-1}dx-\int\limits_{1+\varepsilon}^b\frac{1}{\sqrt{x}}dx.\] 
Let us apply the following substitution in the first case $y=\sqrt{x}-1$. Then we end up with $\int\frac{1}{\sqrt{x}-1}dx=\int\frac{1}{y}2(y+1)dy=2y+2\log y+C=2(\sqrt{x}-1)+2\log(\sqrt{x}-1)+C$. Finally $\int\frac{1}{x-\sqrt{x}}dx=-2+2\log(\sqrt{x}-1)+C$. So we get $F(x)=\frac{(\sqrt{x}-1)^2}{e^2}$. Then $f,f'$ can be obtained easily from that.

Let us verify that it works. $\int\limits_a^b xf'=\frac{1}{e^2}\int\limits_a^b\frac{1}{\sqrt{x}}=\frac{1}{e^2}(\sqrt{b}-\sqrt{a})$. Then 
\[\frac{\int\limits_a^b xf'}{f(b)-f(a)}=\frac{\frac{1}{e^2}(\sqrt{b}-\sqrt{a})}{\frac{1}{e^2}(\frac{1}{\sqrt{a}}-\frac{1}{\sqrt{b}})}=\sqrt{ab}.\qedhere\]
\end{proof}

It is well known that the (ordinary) geometric mean is a quasi-arithmetic mean i.e. the geometric mean can be derived from the arithmetic mean using the $\log$ function: ${\cal{G}}(a,b)=e^{\frac{\log a+\log b}{2}}$. One might ask whether the generalized geometric mean can be derived from $\mathrm{Avg}$ in the same way i.e. whether ${\cal{G}}(H)=e^{\mathrm{Avg}\log H}$ holds where $\log H=\{\log h:h\in H\}$ and $H\subset\mathbb{R}^+$ is a Borel set. The answer is negative as the next proposition shows.

\begin{prp}\label{peavglogneggm}Let $\mu$ be the Borel measure associated to the geometric mean. Then there is a Borel set $H\subset\mathbb{R}^+$ such that $e^{\mathrm{Avg}\log H}\ne\Mm(H)$.
\end{prp}
\begin{proof}Let $H=[1,e^2]\cup[e^4,e^8]$. Then $\log H=[0,2]\cup[4,8]$ hence $\mathrm{Avg}\log H=\frac{1}{2}\frac{2^2-0^2+8^2-4^2}{2+4}=\frac{13}{3}$ which gives that $e^{\mathrm{Avg}\log H}=e^{\frac{13}{3}}$.

\[\Mm(H)=\frac{\sqrt{e^2}-\sqrt{1}+\sqrt{e^8}-\sqrt{e^4}}{\frac{1}{\sqrt{1}}-\frac{1}{\sqrt{e^2}}+\frac{1}{\sqrt{e^4}}-\frac{1}{\sqrt{e^8}}}=\frac{e-1+e^4-e^2}{1-\frac{1}{e}+\frac{1}{e^2}-\frac{1}{e^4}}\ne e^{\frac{13}{3}}.\qedhere\]
\end{proof}

We can go further by proving that the generalized mean ${\cal{M}}(H)=e^{\mathrm{Avg}\log H}$ is not a mean by measure.

\begin{thm}There does not exist a p-Borel measure $\mu$ on $\mathbb{R}^+$ such that $e^{\mathrm{Avg}\log H}=\Mm(H)$ (where $H$ is any Borel set).
\end{thm}
\begin{proof}Clearly $e^{\mathrm{Avg}\log ((a,b))}=\sqrt{ab}\ (a,b\in\mathbb{R}^+)$ because $\mathrm{Avg}((\log a,\log b))=\frac{\log a+\log b}{2}=\log\sqrt{ab}$\ \  (here $(a,b)$ and $(\log a,\log b)$ denote open intervals). If we assumed that $e^{\mathrm{Avg}\log H}$ was a mean by measure, then by \ref{pMmdetbyvoi} we would get that it would be equal to the mean by measure obtained from the (ordinary) geometric mean. But it is false by \ref{peavglogneggm}. 
\end{proof}

One can ask when a generalized mean can be derived as a mean by measure. We can easily provide some sufficient conditions.

\begin{thm}\label{pgenmderasmbm}Let ${\cal M}$ be a generalized mean on the Borel sets of $\mathbb{R}$ and let ${\cal K}(a,b)={\cal M}([a,b])\ (a,b\in\mathbb{R})$ be the associated two variable mean. If ${\cal K}$ determines ${\cal M}$, ${\cal K}$ is strictly internal and ${\cal{K}}(x,y)$ is continuously differentiable by its second variable, then there exists a measure $\mu$ that is absolutely continuous with respect to $\lambda$ such that ${\cal M}=\Mm$.
\end{thm}
\begin{proof}\ref{tmbm}.
\end{proof}

\begin{cor}Let ${\cal M}$ be a generalized mean on the Borel sets of $\mathbb{R}$. If ${\cal M}$ is determined by ${\cal M}|_{{\cal H}'}$ where ${\cal H}'=\{[a',b']:a'<b'\}$, ${\cal M}$ is strict strong internal and the function $f(x,y)={\cal M}([x,y])$ is totally differentiable, then there exists a measure $\mu$ that is absolutely continuous with respect to $\lambda$ such that ${\cal M}=\Mm$.
\end{cor}
\begin{proof}Strict strong internality gives that ${\cal M}([x,y])$ as a two variable mean is strictly internal.
\end{proof}

Now we prove that the inequality between the arithmetic and the geometric mean remains valid for the generalized means too.

\begin{thm}Let $\mu$ be the Borel measure associated to the geometric mean. Then $H\in \mathrm{Dom}(\Mm)$ implies that $\Mm(H)\leq \mathrm{Avg}(H)$.
\end{thm}
\begin{proof} By \ref{suffcondforineq} and \ref{suffcondforcond2} we only have to show that $\frac{g'(x)}{f'(x)}$ is increasing for $g(x)=x-1$ and $f(x)=1-\frac{1}{\sqrt{x}}$. But that is equal to $x\sqrt{x}$.
\end{proof}

\begin{cor}If $I_i=(a_i,b_i)$ and $b_i<a_j$ when $i<j\ (1\leq i,j\leq n)$, then \[\frac{\sum\limits_{i=1}^n\sqrt{b_i}-\sqrt{a_i}}{\sum\limits_{i=1}^n\frac{1}{\sqrt{a_i}}-\frac{1}{\sqrt{b_i}}}\leq
\frac{1}{2}\frac{\sum\limits_{i=1}^n b_i^2-a_i^2}{\sum\limits_{i=1}^n b_i-a_i}.\tag*{\qed}\]
\end{cor}

\begin{ex}Let ${\cal{K}}$ be the harmonic mean: ${\cal{K}}(a,b)=\frac{2}{\frac{1}{a}+\frac{1}{b}}$. Then $F(x)=x-2+\frac{1}{x},f(x)=1-\frac{1}{x^2},f'(x)=\frac{2}{x^3}$. Hence $\mu([a,b])=\mu_f([a,b])=\frac{1}{a^2}-\frac{1}{b^2}$.
\end{ex}
\begin{proof} We have to calculate $\int\frac{1}{x-\frac{2}{1+\frac{1}{x}}}dx=\int\frac{x+1}{x(x-1)}dx=\int\frac{2}{x-1}dx-\int\frac{1}{x}dx=2\log (x-1)-\log x+C=\log\frac{(x-1)^2}{x}+C$ hence $F(x)=x-2+\frac{1}{x}$.

Let us verify it. $\int\limits_a^b xf'=\int\limits_a^b \frac{2}{x^2}=2\frac{b-a}{ab}$. $f(b)-f(a)=\frac{1}{a^2}-\frac{1}{b^2}=\frac{(a+b)(b-a)}{a^2b^2}$. \[\frac{\int\limits_a^b xf'}{f(b)-f(a)}=2\frac{b-a}{ab}\frac{a^2b^2}{(b-a)(a+b)}=2\frac{1}{\frac{1}{a}+\frac{1}{b}}.\qedhere\]
\end{proof}

\begin{ex}Let ${\cal{K}}$ be the logarithmic mean: ${\cal{K}}(a,b)=\frac{a-b}{\log a-\log b}$. Then one can easily verifies that $F(x)=x\log x-x+1,f(x)=\log x,f'(x)=\frac{1}{x}$. Hence $\mu([a,b])=\mu_f([a,b])=\log b-\log a$. \qed
\end{ex}

\subsection{An alternative way}

First let us present another way how the arithmetic mean can be got naturally by some integral.

\medskip

One can easily show that 
\[\int\limits_a^b\int\limits_c^d\frac{x+y}{2}dxdy=\frac{(b^2-a^2)(d-c)+(d^2-c^2)(b-a)}{4}\]
which gives that
\[\frac{\int\limits_a^b\int\limits_c^d\frac{x+y}{2}dxdy}{\lambda([a,b])\lambda([c,d])}=\frac{a+b+c+d}{4}.\]
Hence we end up with
\[\frac{\int\limits_a^b\int\limits_a^b\frac{x+y}{2}dxdy}{\lambda([a,b])^2}=\frac{a+b}{2}.\]

Similarly, for a given two variable mean ${\cal{K}}$ one can try to find a measure $\mu$ on $\mathbb{R}$ such that 
$$\frac{\int\limits_a^b\int\limits_a^b\frac{x+y}{2}d\mu\times\mu}{\mu\times\mu([a,b]\times[a,b])}=
\frac{\int\limits_a^b\int\limits_a^b\frac{x+y}{2}d\mu(x) d\mu(y)}{\mu([a,b])^2}={\cal{K}}(a,b).$$

\begin{thm}\label{tmbm2}Let ${\cal{K}}$ be a two variable mean that is symmetric, strictly internal, continuous and $\frac{\partial{\cal{K}}(x,y)}{\partial y}$ exists for all $(x,y)\in \mathrm{Dom}({\cal{K}})$ and it is continuous. Then there exists a measure $\mu$ that is absolutely continuous with respect to $\lambda$ such that $\displaystyle\frac{\int\limits_a^b\int\limits_a^b\frac{x+y}{2}d\mu(x) d\mu(y)}{\mu([a,b])^2}={\cal{K}}(a,b)$.
\end{thm}
\begin{proof} We follow exactly the same way than in Theorem \ref{tmbm}.

Let us look for $\mu$ in the form $\mu=\mu_f$ where $f$ is an increasing differentiable function. Then $\mu([a,b])=f(b)-f(a)$ and $f'=\frac{d\mu}{d\lambda}$.
Let $F$ be a primitive function of $f$.

Then $$\frac{\int\limits_a^b\int\limits_a^b\frac{x+y}{2}d\mu(x) d\mu(y)}{\mu([a,b])^2}=
\frac{\int\limits_a^b\int\limits_a^b\frac{x+y}{2}f'(x)f'(y)dxdy}{(f(b)-f(a))^2}=$$

$$\frac{\int\limits_a^b\frac{1}{2}f'(y)[bf(b)-af(a)-(F(b)-F(a))+y(f(b)-f(a))]dy}{(f(b)-f(a))^2}=$$

$$\frac{\frac{1}{2}[bf(b)-af(a)-(F(b)-F(a))](f(b)-f(a))}{(f(b)-f(a))^2}+$$

$$\frac{\frac{1}{2}(f(b)-f(a))[bf(b)-af(a)-(F(b)-F(a))]}{(f(b)-f(a))^2}=$$

$$\frac{[bf(b)-af(a)-(F(b)-F(a))](f(b)-f(a))}{(f(b)-f(a))^2}=\frac{bf(b)-af(a)-(F(b)-F(a))}{f(b)-f(a)}.$$

That is exactly the same formula that we got in Theorem \ref{tmbm}. Therefore the same measure will work here as well.
\end{proof}

\begin{prb}One might ask the following question. For a given mean ${\cal{K}}$ can we find a measure $\mu$ on $\mathbb{R}$ such that 
$\displaystyle\frac{\int\limits_a^b\int\limits_a^b {\cal{K}}(x,y) d\mu(x) d\mu(y)}{\mu([a,b])^2}={\cal{K}}(a,b)$? For which means can we expect such measure?
\end{prb}

\section{Behavior in infinity}
It is known that 
$$\lim_{x\to+\infty}\frac{a+b}{2}-\big({\cal H}_p(a+x,b+x)-x\big)=\lim_{x\to+\infty}\frac{(a+x)+(b+x)}{2}-{\cal H}_p(a+x,b+x)=0$$
where ${\cal H}_p(a,b)=\left(\frac{a^p+b^p}{2}\right)^{\frac{1}{p}}$ is the power mean with exponent $p\in\mathbb{R}-\{0\}$, and ${\cal H}_0$ is the geometric mean ($a,b>0$).
I.e. in the far distance a power mean starts to behave as the arithmetic mean. Similarly we can ask when a mean by measure $\Mm$ behaves in the same way, namely
$$\lim_{x\to+\infty}|\Mm(H+x)-\mathrm{Avg}(H+x)|=0\ \ (H\in \mathrm{Dom}(\Mm)).$$
We are going to present a sufficient condition for that.

In this section $\mu$ will denote a p-Borel measure on $\mathbb{R}^+$. 

\begin{df}Let $I\subset\mathbb{R}^+$ be a finite interval.
$$m_I=\inf\left\{\frac{\mu(H)}{\lambda(H)}:H\subset I,H\in \mathrm{Dom}(\Mm)\right\}$$
$$M_I=\sup\left\{\frac{\mu(H)}{\lambda(H)}:H\subset I,H\in \mathrm{Dom}(\Mm)\right\}.$$
\end{df}

\begin{thm}\label{tbii}Let $\mu$ be a p-Borel measure on $\mathbb{R}^+$ such that 

1. if $H$ is $\mu$-measurable, then so is $H+x\ \forall x>0$ and

2. if $I\subset\mathbb{R}^+$ is a finite interval, then $0<m_I\leq M_I<+\infty$ and 
$$\lim_{x\to+\infty}\frac{M_{I+x}}{m_{I+x}}=1.$$ 

Then $H\in \mathrm{Dom}(\Mm)$ implies that $$\lim_{x\to+\infty}|\Mm(H+x)-\mathrm{Avg}(H+x)|=0.$$
\end{thm}
\begin{proof} Let $H\in \mathrm{Dom}(\Mm), x>0$. Let $I\subset\mathbb{R}^+$ be a finite interval such that $H\subset I$.

First let us observe that $H\in \mathrm{Dom}(\Mm)$ implies that $H+x\in \mathrm{Dom}(\Mm) \ \forall x>0$ by the first condition.

Then we get the statement by
\[\frac{m_{I+x}}{M_{I+x}}\mathrm{Avg}(H+x)=
\frac{m_{I+x}\int\limits_{H+x} x d\lambda}{M_{I+x}\lambda(H+x)}\leq
\frac{\int\limits_{H+x} x d\mu}{\mu(H+x)}\leq\]
\[\frac{M_{I+x}\int\limits_{H+x} x d\lambda}{m_{I+x}\lambda(H+x)}=
\frac{M_{I+x}}{m_{I+x}}\mathrm{Avg}(H+x).\qedhere\]
\end{proof}

Now our aim is to prove that the geometric mean satisfies these conditions.

\begin{lem}\label{lcsm}Let $(H_i)$ is a sequence of $\mu$-measurable sets such that $H_i\to H$ in the pseudo-metric $d_{\mu}$. Then $\mu(H_i)\to\mu(H)$. \qed
\end{lem}

\begin{lem}\label{lmic}Let $\mu$ be absolutely continuous with respect to $\lambda$. Let $(H_i)$ be a sequence of bounded $\lambda$-measurable sets such that $H_i\to H$ in the pseudo-metric $d_{\lambda}$. Let $\lambda(H)>0, H$ be bounded. Then $\frac{\mu(H_i)}{\lambda(H_i)}\to\frac{\mu(H)}{\lambda(H)}$. 
\end{lem}
\begin{proof} By absolute continuity $H_i\to H$ according to $d_{\mu}$ as well. Then by Lemma \ref{lcsm} $\mu(H_i)\to\mu(H)$ and $\lambda(H_i)\to\lambda(H)$.
\end{proof}

\begin{lem}\label{lmini}Let $\mu$ is absolutely continuous with respect to $\lambda$ and vica versa. Let $I\subset\mathbb{R}^+$ be a finite interval. Then
\[m_I=\inf\bigg\{\frac{\mu(K)}{\lambda(K)}:K=\bigcup_{i=1}^n I_i,  I_i\subset I\text{ is an interval}, I_i\cap I_j=\emptyset\ (i\ne j),\]
\[I_i\in \mathrm{Dom}(\Mm)\ (i,j\in\mathbb{N})\bigg\}.\]
A similar statement holds for $M_I$ with $\sup$.
\end{lem}
\begin{proof} First let us observe that $\mathrm{Dom}(\Mm)=\mathrm{Dom}({\cal{M}}^{\lambda})$.

Let $m$ denote the right hand side of the above equality. Obviously $m_I\leq m$.

Let $H\subset I,H\in \mathrm{Dom}(\Mm),\varepsilon>0$. Then by \cite{billings} 12.3 there are countably many disjoint (open) intervals $(I_i)$ such that $H\subset\cup_1^{\infty}I_i$ and $\sum_1^{\infty}\lambda(I_i)<\lambda(H)+\varepsilon$. Then we can choose $n\in\mathbb{N}$ such that $\lambda(H)-\varepsilon<\sum_1^{n}\lambda(I_i)<\lambda(H)+\varepsilon$. Therefore we can construct a sequence $(K_i)$ such that $K_i$ is a disjoint union of finitely many intervals and $\lambda(K_i)\to\lambda(H)$. By Lemma \ref{lmic} $\frac{\mu(K_i)}{\lambda(K_i)}\to\frac{\mu(H)}{\lambda(H)}$.

The proof of the statement for $M_I$ is analogous.
\end{proof}

\begin{prp}\label{pgeomsatl}The Borel measure $\mu$ associated to the geometric mean satisfies the conditions of Theorem \ref{tbii}.
\end{prp}
\begin{proof} Let $I=[c,d]$. Let us calculate $m_{I+x}$.

If $a,b\in I+x,a<b$, then $\mu([a,b])=\frac{1}{e^2}\left(\frac{1}{\sqrt{a}}-\frac{1}{\sqrt{b}}\right)=\frac{1}{e^2}\frac{b-a}{\sqrt{ab}(\sqrt{a}+\sqrt{b})}$. We get that $\frac{\mu([a,b])}{\lambda([a,b])}=\frac{1}{e^2}\frac{1}{\sqrt{ab}(\sqrt{a}+\sqrt{b})}$. If we want its infimum for $a,b$, then $a,b$  have to tend to $d+x$. Therefore for one interval the infimum is $\frac{1}{e^2}\frac{1}{2(d+x)\sqrt{d+x}}$.

We want to show that if $I_1,\dots,I_n\subset I,\ \sup I_i<\inf I_j\ (i<j)$, then $\frac{\mu(I_1)+\dots+\mu(I_n)}{\lambda(I_1)+\dots+\lambda(I_n)}>\frac{\mu(I_n)}{\lambda(I_n)}$. We go by induction: 
Let $I_1,I_2\subset I,\ \sup I_1<\inf I_2$. Then evidently $\frac{\mu(I_1)+\mu(I_2)}{\lambda(I_1)+\lambda(I_2)}>\frac{\mu(I_2)}{\lambda(I_2)}$ holds because $\frac{\mu(I_1)}{\lambda(I_1)}>\frac{\mu(I_2)}{\lambda(I_2)}$. 
Let us assume that $\frac{\mu(I_1)+\dots+\mu(I_{n-1})}{\lambda(I_1)+\dots+\lambda(I_{n-1})}>\frac{\mu(I_{n-1})}{\lambda(I_{n-1})}$ holds whenever $\sup I_i<\inf I_j\ (i<j)$. Clearly $\frac{\mu(I_1)+\dots+\mu(I_n)}{\lambda(I_1)+\dots+\lambda(I_n)}>\frac{\mu(I_n)}{\lambda(I_n)}$ is equivalent with $\frac{\mu(I_1)+\dots+\mu(I_{n-1})}{\lambda(I_1)+\dots+\lambda(I_{n-1})}>\frac{\mu(I_n)}{\lambda(I_n)}$ but by $\frac{\mu(I_{n-1})}{\lambda(I_{n-1})}>\frac{\mu(I_n)}{\lambda(I_n)}$ it holds.

Hence by Lemma \ref{lmini} we get that $m_{I+x}=\frac{1}{e^2}\frac{1}{2(d+x)\sqrt{d+x}}$. Similarly $M_{I+x}=\frac{1}{e^2}\frac{1}{2(c+x)\sqrt{c+x}}$. Finally
\[\lim_{x\to+\infty}\frac{M_{I+x}}{m_{I+x}}=\lim_{x\to+\infty}\frac{(d+x)\sqrt{d+x}}{(c+x)\sqrt{c+x}}=
\lim_{x\to+\infty}\frac{(\frac{d}{x}+1)\sqrt{\frac{d}{x}+1}}{(\frac{c}{x}+1)\sqrt{\frac{c}{x}+1}}=1.\qedhere\]
\end{proof}

\begin{thm}Let $\mu$ be the Borel measure associated to the geometric mean. Then $H\in \mathrm{Dom}(\Mm)$ implies that \[\lim_{x\to+\infty}|\Mm(H+x)-\mathrm{Avg}(H+x)|=0.\]
\end{thm}
\begin{proof}\ref{pgeomsatl}, \ref{tbii}.
\end{proof}


{\footnotesize


\noindent 
email: alosonczi1@gmail.com\\
}

\end{document}